\documentclass{article} 

\usepackage{natbib}
\usepackage{algorithm}
\usepackage{algorithmic}
\usepackage{times}
\usepackage{amsmath,amsthm,amsfonts,amssymb,mathrsfs}
\usepackage[fleqn,tbtags]{mathtools}
\usepackage{multirow}
\usepackage{url}
\usepackage{graphicx} 
\usepackage{hyperref}

\theoremstyle{plain}
\newtheorem{theorem}{Theorem}[section]
\newtheorem{proposition}[theorem]{Proposition}
\newtheorem{lemma}[theorem]{Lemma}
\newtheorem{corollary}[theorem]{Corollary}

\theoremstyle{definition}

\theoremstyle{remark}

\def\eop{$\rule{1.3ex}{1.3ex}$}
\renewcommand\qedsymbol\eop  
\numberwithin{equation}{section}
\makeatletter

\DeclareMathOperator\proj{Proj}
\DeclareMathOperator\prox{prox}
\DeclareMathOperator\dom{dom}

\newcommand{\argmin}[1]{\underset{#1}{\operatorname{argmin}}}

\newcommand{\beq}{\begin{equation}} \newcommand{\eeq}{\end{equation}}

\newcommand{\bi}{\begin{itemize}}
\newcommand{\be}{\begin{enumerate}}
\newcommand{\ei}{\end{itemize}}
\newcommand{\ee}{\end{enumerate}}

\newcommand{\R}{{\mathbb R}}
\newcommand{\N}{{\mathbb N}}

\newcommand{\calA}{{\cal A}}

\newcommand{\lb}{{\langle}}
\newcommand{\rb}{{\rangle}}

\newcommand{\Sb}{{\bf S}}

\def\boldf#1{\hbox{\rlap{$#1$}\kern.4pt{$#1$}}}

\newcommand{\trans}{^{\scriptscriptstyle \top}}

\mathtoolsset{showonlyrefs,showmanualtags}

\newcommand{\rd}{\R^d}
\newcommand{\phii}{\varphi}

\newcommand{\scO}{\mathcal{O}}

\newcommand{\moreau}[2]{{#1}_{#2}}
\newcommand{\betaphii}{\moreau{\phii}{\beta}}
\newcommand{\betaphiiprime}{\moreau{\phii}{\beta'}}

\begin{document} 

\title{PRISMA: PRoximal Iterative SMoothing Algorithm}

\author{
Francesco Orabona \\
Toyota Technological Institute at Chicago\\
\texttt{francesco@orabona.com} \\
\and
Andreas Argyriou\\
Katholieke Universiteit Leuven \\
\texttt{andreas.argyriou@esat.kuleuven.be} \\
\and
Nathan Srebro\\
Toyota Technological Institute at Chicago\\
\texttt{nati@ttic.edu}
}

\maketitle

\newcommand{\nati}[1]{{\color{red}\bf [[ {#1} -Nati ]]}}
\newcommand{\removedbynati}[1]{}

\newcommand{\fix}{\marginpar{FIX}}
\newcommand{\new}{\marginpar{NEW}}

\begin{abstract}
Motivated by learning problems including max-norm regularized matrix
completion and clustering, robust PCA and sparse inverse covariance selection, we propose a
novel optimization algorithm for minimizing a convex objective which
decomposes into three parts: a smooth part, a simple non-smooth
Lipschitz part, and a simple non-smooth non-Lipschitz part.  
We use a time variant smoothing strategy that 
allows us to obtain a guarantee that does not depend on knowing
in advance the total number of iterations nor a bound on the domain.
\end{abstract}

\section{Introduction}
\label{sec:intro}

We propose an optimization method (PRISMA --- PRoximal Iterative
SMoothing Algorithm) for problems of the form\footnote{All the
  theorems hold also in general Hilbert spaces, but for
  simplicity of exposition we consider a Euclidean setting.}:
\begin{equation}
\min\left\{
F(x):=f(x) + g(x) + h(x) : x \in \R^n
\right\}
\label{eq:opt}
\end{equation}
where:
\begin{itemize}
\item $f:\R^n\rightarrow\R$ is a convex and $L_f$-smooth function, that is,
  differentiable with a Lipschitz continuous gradient: $ \|\nabla f(x)
  - \nabla f(y) \| \leq L_f \, \|x-y\|  \quad\quad \forall x,y \in \R^n$.
\item $g:\R^n\rightarrow\R$ is a convex $\rho_g$-Lipschitz continuous function:
  $|f(x)-f(y)|\leq \rho_g \|x-y\|$.
\item $h:\R^n\rightarrow\R\cup\{+\infty\}$ is a proper, lower semicontinuous, convex (but possibly
  non-continuous/non-finite) function. For example, $h$ could be an indicator
  function for a convex constraint.
\end{itemize}
We further assume that we can calculate gradients of $f$, and that $g$
and $h$ are ``simple'' in the sense that we can calculate their {\em
  proximity operators}\footnote{For
  more clarity here we explicitly define the proximity operator as a
  function of $\alpha$ too because our algorithm heavily depends on
  this parameter changing over time.}:
\begin{align}
\prox_g(x,\alpha)=\argmin{u} \left\{ \frac{\|x-u\|^2}{2 \alpha} + g(u) :u\in\R^n \right\},
\end{align}
where $x \in \R^n$, $\alpha\in\R_{++}$, and similarly for $h$.  That is, our method can be
viewed as a black-box one, which accesses $f,g$ and $h$ only
through the gradients of $f$ and the proximity operators of $g$ and
$h$. 
Each iteration of PRISMA requires one evaluation of each of $\nabla f$,
$\prox_g$ and $\prox_h$ and a few vector operations of overall
complexity $O(n)$, and after $k$ such iterations we have
\begin{align}\label{eq:introbound}
F(x_{k+1}) - F(x^*) = \scO\left(\frac{L_f}{k^2} +  \frac{\rho_g \log k}{k}\right)~.
\end{align}

\paragraph{Applications.} Our main
motivation for developing PRISMA was for solving optimization problems
involving the matrix max-norm (aka
$\gamma_{2:\ell_1\rightarrow\ell_\infty}$ norm).  The max-norm has
recently been shown to possess some advantages over the more commonly
used trace-norm \cite{lee_maxnorm,jalali_maxnorm,RinaNatiCOLT}, but
is lagging behind it in terms of development of good
optimization approaches---a gap we aim to narrow in this paper.  Both
norms are SDP representable, but standard SDP solvers are not
applicable beyond very small scale problems. Instead, several
first-order optimization approaches are available and commonly used
for the trace-norm, including singular value thresholding \cite{cai} and
other \cite{jaggi2010simple,ma,TomiokaSSK10}. However, until now there have been no practical
first-order optimization techniques available for the max-norm. 
In Sec. \ref{sec:app} we show how max-norm regularized problems can be
written in the form \eqref{eq:opt} and solved using PRISMA, thus
providing for the first time an efficient and practical first order
method for max-norm optimization.  

We also demonstrate how PRISMA can be applied also to other
optimization problems with complex structure, including robust
principal component analysis (robust PCA), sparse inverse covariance
selection and basis pursuit.  In particular, for robust PCA, we show
how PRISMA yields better performance then previously published
approaches, and for basis pursuit we obtain the best known convergence
rate using only first-order and proximity oracles.

\textbf{Our Approach.}
Following ideas of Nesterov
\cite{nesterov2005smooth} and others, we propose to smooth the function $g$
and use its proximity operator to obtain gradients of its smoothed
version. However, unlike \citet{nesterov2005smooth}, where a fixed
amount of smoothing is used throughout the run, we show
how to gradually change the amount of smoothing at each iteration.
We can then use an accelerated gradient descent approach on
$f$ plus the smoothed version of $g$, as in
\citet{nesterov2005smooth}.  We also use the ideas of partial linearization to
handle the component $h$: instead of attempting to linearize it as
in gradient descent, we include it as is in each iteration, as in
FOBOS \cite{duchi} and ISTA/FISTA \cite{fista}.
The gradual smoothing allows us to obtain a guarantee that does
not depend on knowing in advance the total number of iterations,
needed in \citet{nesterov2005smooth}, nor a bound on the domain, needed
in \citet{nesterov2005smooth,ChambolleP11}, and paying only an additional $\log(k)$ factor.


{\bf Notation.} We use $f^*$ to denote the Fenchel conjugate,
$\delta_Q$ the indicator of the set $Q$ (zero inside $Q$ and infinite
outside), $\Sb^n_{++}$ the set of $n \times n$ positive definite
matrices, and $\Sb^n_{+}$ the set of p.s.d. ones.

\section{Smoothing of Nonsmooth Functions}
\label{sec:smoothing}

As was suggested by \citet{Nesterov05,nesterov2005smooth} and
others, in order to handle the non-smooth component $g$, we 
approximate it using a smooth function.  Such a smoothing plays a
central role in our method, and we devote this section to carefully
presenting it.  In particular, we use the $\beta$-{\em Moreau
  envelope}~\cite{combettes_book}, known also as {\em Moreau-Yosida
  regularization}. For a function $\phii:\R^n \to \R\cup\{+\infty\}$,
the Moreau envelope $\betaphii$, where $\beta > 0$, is defined as
\begin{equation}
{\betaphii}(x) := \inf \left\{ \frac{\|x-u\|^2}{2\beta} + \phii(u) :u\in\R^n \right\} \  \forall x \in \R^n~.
\label{eq:moreau_env}
\end{equation}
The function ${\betaphii}$ is a smooth approximation to $\phii$ (in
fact, the best possible approximation of bounded smoothness), as summarized in
the following Lemmas:
\begin{lemma}[Proposition 12.29 in \citet{combettes_book}]
\label{lem:moreau_props}
Let $\phii:\R^n \to \R$ be a proper, lower semicontinuous, convex function and $\beta > 0$. Then
${\betaphii}$ is $\frac{1}{\beta}$-smooth and its gradient can be obtained from the proximity operator of $\phii$ as: $\nabla (\betaphii)(x) = \frac{1}{\beta} (x - \prox_{\phii}(x,\beta))$.
\end{lemma}
\begin{lemma}
\label{lem:moreau_error}
Let $\phii:\R^n \to \R$ a convex, $\rho$-Lipschitz function, then
\begin{enumerate}
\item if $\beta > 0$, then $\betaphii \leq \phii \leq \betaphii + \frac{1}{2}\beta\rho^2$;
\item if $\beta \geq \beta'>0$, then  $\ \betaphii \leq \betaphiiprime \leq \betaphii + \frac{1}{2}(\beta-\beta')\rho^2$.
\end{enumerate}
\end{lemma}
\begin{proof}
The proof extends the one of Lemma 2.5 in \citet{trading} to vectorial functions. Define $\Psi_x(u)= \frac{1}{2\beta}\|x-u\|^2 + \phii(u)$.
We have $\betaphii(x) = \inf_u \Psi_x(u) \leq \Psi_x(x) = \phii(x)$, and this proves the left hand side of the first property.
For the other side of the inequality, we have
\begin{equation}
\begin{split}
\Psi_x(u) &= \frac{1}{2\beta}\|x-u\|^2 + \phii(u) - \phii(x) + \phii(x) \\
&\geq \frac{1}{2\beta}\|x-u\|^2 + \phii(x) - \rho |x-u| 
\end{split}
\end{equation}
where we have used the Lipschitz property of $\phii$. Hence
\begin{equation}
\begin{split}
\betaphii(x) &\geq \phii(x) + \inf_u \left(\frac{\|x-u\|^2}{2\beta} - \rho |x-u| \right)\\
&= \phii(x) - \frac{1}{2}\beta\rho^2~.
\end{split}
\end{equation}
The second property follows from the first one and $\moreau{(\moreau{\phii}{\beta'})}{\beta-\beta'}= \moreau{\phii}{\beta}$
(Proposition 12.22 in \citet{combettes_book}).
\end{proof}
These two lemmas show how to obtain a smooth
approximation of a Lipschitz continuous function, with the tradeoff
between the smoothness and the degree of approximation controlled by
the Lipschitz constant of the non-smooth function.  Furthermore, in order to be able to work with the smoothed approximation
$\betaphii$, and in particular calculate its gradients, all that is
required is access to the proximity operator of the non-smooth
function $\phii$.  This is the reason we require access to $\prox_g$.

\textbf{Relation to Nesterov Smoothing.} The Moreau envelope
described above also underlies Nesterov's smoothing method
\cite{nesterov2005smooth}, although his derivation is somewhat
different. Nesterov refers to a non-smooth function $\phii$ that can
be written as:
\begin{equation}
\phii(x)  = \max \{\langle x, u \rangle -\hat{\phii}(u) : u \in Q \}
\label{eq:phi}
\end{equation}
where $x \in \R^n$, $Q\subseteq\R^m$ is a bounded closed convex set, and proposes to
smoothen the function $\phii\circ A$, where $A:\R^n \to \R^m$ is a
linear operator, with $\tilde{\phii}_\beta\circ A$, where:
\begin{equation}
\tilde{\phii}_\beta(x):=\max \{\langle x, u \rangle -\hat{\phii}(u) - \beta d(u): u \in Q \},
\label{eq:nest}
\end{equation}
with $x\in\R^n$, $\beta>0$, and $d(u) \geq \frac{\sigma}{2} \|u-u_0\|^2$.  The
smoothness and approximation properties of $\tilde{\phii}_{\beta}$ are then
given in terms of the diameter of the set $Q$, the parameter $\sigma$
and a certain norm of the matrix $A$.
Any convex function $\phii$ can indeed be written in the form
\eqref{eq:phi}, where $\hat{\phii}$ is its convex conjugate and $Q$ is
the set of all subgradients of $\phii$ --- the diameter of $Q$ thus
corresponds to the Lipschitz constant (bound on the magnitude of the
subgradients) of $\phii$.  Focusing on $A$ being the identity and
$d(u)=\frac{1}{2}\| u \|^2$ (i.e. $\sigma=1$), we have that
$\tilde{\phii}_{\beta}$ defined in \eqref{eq:nest} is exactly the Moreau
envelope $\betaphii$:
\begin{proposition}
  If $\phii$ is a proper, lower semicontinuous convex function,
  $\hat{\phii}$ is its Fenchel conjugate, $Q$ the set of all its
  subgradients and $d(u)=\frac{1}{2}\| u \|^2$, then
  $\tilde{\phii}_{\beta}=\betaphii$.
\end{proposition}
This is obtained as a corollary from a slightly more general Lemma:
\begin{lemma}
  If $\phii$ is a proper, lower semicontinuous convex function which
  can be written as \eqref{eq:phi} and $d(u) = \frac{1}{2} \|u\|^2$,
  then $\tilde{\phii}_\beta=\moreau{(\phii\,\square\, \delta_Q^*)}{\beta}$, where
  $f\,\square\, g$ is the infimal convolution of $f$ and $g$.
\end{lemma}
\begin{proof}
From the definition of $\phii$ and \citep[Thm. 13.32]{combettes_book}, 
we derive that $\hat{\phii} = \phii^*-\delta_Q$. Hence
\begin{equation}
\begin{split}
\tilde{\phii}_\beta(x) &= \max_{u \in Q}\left\{  \langle x, u \rangle -\phii^*(u) - \frac{\beta}{2} \|u\|^2 \right\} \\
&= \left(\phii^*+\frac{\beta}{2} \|\cdot\|^2 + \delta_Q \right)^*(x)  = \moreau{(\phii\,\square\, \delta_Q^*)}{\beta}(x),
\end{split}
\end{equation}
where the last equality follows from \citep[Thm. 13.32, Prop. 13.21 and Prop. 14.1]{combettes_book}.
\end{proof}
Nesterov's formulation is somewhat more general, allowing for
different regularizers $d(u)$ and different sets $Q$, but our
presentation is a crisper and more direct statement of the assumptions on the function $\phii$ to be
smoothed and its relationship to $\betaphii$: we just need to be able to
calculate the proximity operator of $\phii$, to obtain gradients of
$\betaphii$, and we need to rely on its Lipschitz constant in order to
be able to control the quality of the approximation. 

\begin{algorithm}[t]
\caption{PRISMA}
\begin{algorithmic}
\STATE {\bf Parameters}~ $\{\beta_k >0: k\in\N\}$
\STATE {\bf Initialize}~ $x_1=y_1 \in\dom h$, $\theta_1=1, L_1=L_f+\frac{1}{\beta_1}$
\FOR {$k=1,2,\ldots$ }
\STATE $L_{k+1} \leftarrow L_f+\frac{1}{\beta_{k+1}}$
\STATE $\theta_{k+1} \leftarrow \frac{2}{1+\sqrt{1+\frac{4 L_{k+1}}{\theta^2_k\, L_k}}}$
\STATE $x_{k+1} \leftarrow \prox_{h}
\Big((1 - \tfrac{1}{L_k\beta_k})y_k - \frac{1}{L_k}\nabla f(y_k)
 + \tfrac{1}{L_k\beta_k}\prox_{g}(y_k,\beta_k )\ , \  \frac{1}{L_k}
\Big)$
\STATE $y_{k+1} \leftarrow  x_{k+1} +
\theta_{k+1}\left(\frac{1}{\theta_k}-1\right) (x_{k+1}-x_k)$
\ENDFOR 
\end{algorithmic}
\label{alg:fista}
\end{algorithm}

\section{PRoximal Iterative SMoothing Algorithm}
\label{sec:fista}

Our proposed method, PRISMA, is given as Algorithm~\ref{alg:fista}.
As is standard in ``accelerated'' first order methods (e.g. \cite{fista}),
we keep track of two sequences of
iterates, $x_k$ and $y_k$. At each iteration we perform a proximal
gradient update $x_{k+1} \leftarrow \prox_{h} \left( y_k -
  \frac{1}{L_k} \nabla(f+\moreau{g}{\beta_k})(y_k) \ , \ \frac{1}{L_k}
\right)$, where the gradient of $\moreau{g}{\beta_k}$ is calculated
according to Lemma \ref{lem:moreau_props} and $\beta_k$ is some
sequence of parameters. We then set $y_{k+1}$ to be a linear
combination of $x_{k+1}$ and $x_k$.  
Note also that the recursive formula for the sequence $\theta_k$ 
satisfies 
\beq
\frac{1}{\theta_{k+1}^2}-\frac{1}{\theta_{k+1}} =
\frac{L_{k+1}}{L_k} \frac{1}{\theta_k^2} \;.
\label{eq:theta}
\eeq

We establish the following guarantee on the values of the objective function.
\begin{theorem}
\label{thm:fista}
If the sequence $\beta_k$ is nonincreasing,
then the iterates of Algorithm~\ref{alg:fista} satisfy, $\forall x^*\in \dom h$,
\begin{equation}
\begin{split}
F(x_{k+1}) - F(x^*)  \leq \frac{1}{2} \theta_k^2 L_k  \left( \|x^*-x_1\|^2 +
 \rho_g^2\ \sum_{i=1}^k \frac{\beta_i}{\theta_i L_i} \right)~.
\label{eq:fista}
\end{split}
\end{equation}
\end{theorem}

In the following we show two possible ways to set the sequences
$\beta_k$. One possibility would be to decide on a fixed number of iterations $T$ and to set $\beta_k$ to a
constant value.
\begin{corollary}
\label{cor:fixed_beta}
For a given $T>0$ set $\beta_k=\frac{2\|x_1 - x^*\|}{\rho_g(T+1)} \  \forall k$.
Then for every $x^*\in \dom h$
\begin{equation}
F(x_{T+1}) - F(x^*) \leq 2 \frac{L_f \|x^*-x_1\|^2}{(T+1)^2} +  2\frac{\rho_g \|x^*-x_1\|}{T+1}~.
\end{equation}
\end{corollary}
This choice is optimal if we know the number of iterations $T$ we will
use, as well as the Lipschitz constant and a bound on the distance from a minimizer,
sharing the same limitation with \citet{nesterov2005smooth}.  But with
this fixed choice of smoothing, even if we perform additional
iterations, we will not converge to the optimum of our original
objective (only of the smoothed, thus approximate, objective).

We propose a dynamic setting of $\beta_k$, that gives a bound
that holds uniformly for any number of iterations, and thus a method
which converges to the optimum of the original objective.
\begin{corollary}
\label{cor:adaptive}
Let $a > 0$ and 
set $\beta_k=\frac{1}{a k}$ in Algorithm~\ref{alg:fista}, then, uniformly for any $k \in \N$ and for every $x^*\in \dom h$ 
\begin{equation}
\begin{split}
F(x_{k+1}) - F(x^*)  \leq 2\frac{L_f+a k}{(k+1)^2} \Bigg[ \|x^*-x_1\|^2 
 + \frac{\rho_g^2}{a} \Biggl( &
\frac{3}{2a} \log \frac{L_f + a k}{L_f+a} 
 + \frac{1}{L_f+a}
 \Biggr)\Bigg]~.
\end{split}
\end{equation}
\end{corollary}
\begin{proof}
It is easy to derive that $\frac{1}{k+1} < \theta_k\leq \frac{2}{k+1}$ (by induction) and the
elementary chain of inequalities $\sum_{i=1}^k \frac{1}{a+b i} =
\frac{1}{a+b}+\sum_{i=2}^k \frac{1}{a+b i} \leq \frac{1}{a+b}+\int_1^k \frac{1}{a+b
  x} dx$.
The assertion then follows from Theorem~\ref{thm:fista} and these facts.
\end{proof}

Note that the optimal choice of $a$ does depend on the Lipschitz constant $\rho_g$ 
and on the distance of our initial point
from a minimizer, but that the algorithm does converge for \emph{any} choice of $a$.
Thus, compared to the bound in Corollary~\ref{cor:fixed_beta}, the price we pay 
for not knowing in advance how many iterations 
will be needed is only an additive logarithmic factor\footnote{\citet{OuyangG12} 
use a similar proof technique, for 
the simpler case when $h=0$, and they report a convergence rate bound 
without the logarithmic term. Unfortunately their 
proof contains an error; they fixed it in the arxiv version of their paper~\cite{OuyangG12arxiv}. Their new theorem contains the 
term $\max_k \|x^* - x_k\|^2$, which is not known how it can be bounded.}.

When $f=0$ the updates reduce to $x_{k+1} \leftarrow \prox_{h} \left( \prox_{g}(y_k,\frac{1}{a\,k}),\frac{1}{a\,k} \right)$, and $\theta_k=\frac{1}{k}$. 
This can be viewed as an adaptive version of an accelerated backward-backward splitting algorithm.

Before proving Theorem \ref{thm:fista}, we need two additional technical lemmas.
\begin{lemma}[{\em descent lemma}, Thm. 18.15 in \citet{combettes_book}, Thm. 2.1.5 in \citet{nesterov_book}]
The convex function $f$ is $L$-smooth if and only if 
\begin{equation}
f(x) \leq f(y) + \lb x-y, \nabla f(y) \rb + \frac{L}{2} \|x-y\|^2 \quad  \forall x,y \in \R^n.
\label{eq:descent}
\end{equation}
\label{lem:descent}
\end{lemma}
\begin{lemma}[Thm. 2.1.2 in \citet{nesterov_book}]
For a $\mu$-{\em strongly convex} function $g$, and ${\hat v}=\argmin{}\, g$,
\begin{align}
&& g(w) - g({\hat v}) \geq \frac{\mu}{2} \|w-{\hat v}\|^2 
&& \forall w \in \R^n \,.
\label{eq:strong}
\end{align}
\label{lem:strong}
\end{lemma}

\begin{proof}\textbf{(of Theorem~\ref{thm:fista})}
Denote $F^{\beta_k}(x):=f(x)+\moreau{g}{\beta_k}(x)+h(x)$.
Fix an arbitrary $k\in\N$. Since $f + \moreau{g}{\beta_k}$ is
$L_k$-smooth by Lemma~\ref{lem:moreau_props}, applying Lemma~\ref{lem:descent} yields
\begin{align*}
F^{\beta_k}(x_{k+1}) &\leq f(y_k) + \moreau{g}{\beta_k}(y_k) + \frac{L_k}{2} \|x_{k+1}-y_k\|^2 \\
& + \left\langle \nabla (f + \moreau{g}{\beta_k})(y_k), x_{k+1}-y_k \right\rangle + h(x_{k+1})~.
\end{align*}
We now use the strong convexity of the function
$x\mapsto\left\langle \nabla (f + \moreau{g}{\beta_k})(y_k), x\right\rangle + \tfrac{L_k}{2} \|x-y_k\|^2 + h(x)$,
whose minimizer equals $x_{k+1}$ (by Lemma~\ref{lem:moreau_props}).
Using Lemma \ref{lem:strong} with $w = (1-\theta_k)x_k + \theta_k x$, 
we obtain
\begin{equation}
\begin{split}
F^{\beta_k}(x_{k+1}) &\leq f(y_k) + \moreau{g}{\beta_k}(y_k) + h\left( (1-\theta_k)x_k + \theta_k x \right) \\
& \quad +\lb \nabla (f+ \moreau{g}{\beta_k})(y_k), (1-\theta_k)x_k + \theta_k x - y_k\rb \\
& \quad + \frac{L_k}{2} \|(1-\theta_k)x_k + \theta_k x -y_k\|^2  
 - \frac{L_k}{2} \|(1-\theta_k)x_k + \theta_k x - x_{k+1}\|^2 \,.
\end{split}
\end{equation}
We let $z_k := \frac{1}{\theta_k}y_k + \left(1-\frac{1}{\theta_k}\right)x_k$
for every $k\in\N$ and note that $z_{k+1} = x_k + \frac{1}{\theta_k}(x_{k+1}-x_k)$
by the algorithmic construction of $y_{k+1}$. Therefore,
\begin{equation}
\begin{split}
F^{\beta_k}(x_{k+1}) &\leq f(y_k) + \moreau{g}{\beta_k}(y_k) + h\left( (1-\theta_k)x_k + \theta_k x \right) \\
&  \quad + \lb \nabla (f + \moreau{g}{\beta_k})(y_k), (1-\theta_k)x_k + \theta_k x - y_k\rb \\
&  \quad + \frac{\theta_k^2 L_k}{2} \|x-z_k\|^2 - \frac{\theta_k^2 L_k}{2} \|x-z_{k+1}\|^2 \\
&\leq f(y_k) + \moreau{g}{\beta_k}(y_k) + (1-\theta_k) h(x_k) + \theta_k h(x)\\
&  \quad + \lb \nabla (f + \moreau{g}{\beta_k})(y_k), (1-\theta_k)x_k + \theta_k x - y_k\rb \\
&  \quad + \frac{\theta_k^2 L_k}{2} \|x-z_k\|^2 - \frac{\theta_k^2 L_k}{2} \|x-z_{k+1}\|^2 \\
&\leq (1-\theta_k) F^{\beta_k}(x_k) + \theta_k F^{\beta_k}(x) \\
&  \quad + \frac{\theta_k^2 L_k}{2} \|x-z_k\|^2 - \frac{\theta_k^2 L_k}{2} \|x-z_{k+1}\|^2,
\end{split}
\end{equation}
where we have used the convexity of the functions $f, g, h$.
Using Property 1 in Lemma \ref{lem:moreau_error} we have
\begin{align*}
F^{\beta_k}(x_{k+1}) - F(x) &\leq \frac{\theta_k^2 L_k}{2} \left(\|x-z_k\|^2 - \|x-z_{k+1}\|^2\right) \\ 
& \quad + (1-\theta_k) \left( F^{\beta_k}(x_k) - F(x) \right)~.
\end{align*}
We now use Property 2 in Lemma~\ref{lem:moreau_error} to change the smoothing parameter. Denoting by $D_k=F^{\beta_{k}}(x_{k}) - F(x)$, we obtain
\begin{equation}
\begin{split}
\frac{D_{k+1}}{\theta_k^2 L_k} &\leq  \frac{1}{2} \left(\|x-z_k\|^2 - \|x-z_{k+1}\|^2\right) 
 +\frac{(\beta_k-\beta_{k+1})\rho_g^2}{2 \theta_k^2 L_k} + \frac{1-\theta_k}{\theta_k^2 L_k} D_{k} ~.
\end{split}
\end{equation}
Using the definition of $\theta_k$ \eqref{eq:theta}, and summing we obtain
\begin{equation}
\begin{split}
\frac{D_{k+1}}{\theta_k^2 L_k}
&\leq  \frac{1}{2} \left(\|x-z_1\|^2 - \|x-z_{k+1}\|^2\right) 
 + \sum_{i=1}^k \frac{(\beta_i-\beta_{i+1})\rho_g^2}{2 \theta_i^2 L_i}~.\\
\end{split}
\end{equation}


Finally, we apply Property 1 in Lemma \ref{lem:moreau_error} to obtain 
$F(x_{k+1}) \leq F^{\beta_{k+1}}(x_{k+1}) + \frac{1}{2} \beta_{k+1} \rho_g^2$
and reordering the terms we obtain that
\begin{equation}
\begin{split}
D_{k+1}
&\leq  \frac{1}{2} \theta_k^2 L_k  \|x-x_1\|^2 +
\frac{1}{2} \beta_{k+1} \rho_g^2  + 
\frac{1}{2} \theta_k^2 L_k \sum_{i=1}^k \frac{(\beta_i-\beta_{i+1}) \rho_g^2}{\theta^2_i L_i} \;.
\end{split}
\end{equation}
Applying the definition of $\theta_k$ \eqref{eq:theta} 
and gathering the $\beta_i$ terms, we obtain the stated bound.
\end{proof}




\subsection*{Relation to Prior Work}  

With $g=0$, PRISMA with fixed smoothing reduces to a
variant of FISTA \citep{fista}, with essentially the same guarantee of
$\scO(\frac{L_f}{k^2})$. 

Conversely, with $h=0$, or $h$ being an indicator for a convex domain,
PRISMA  becomes similar to Nesterov's smoothed accelerated method \citep{nesterov2005smooth},
obtaining the rate of $\scO(\frac{\log k}{k})$, but with some important differences.  First, we provide
an explicit bound in terms of the Lipschitz constant of $g$, as
discussed in Sec.~\ref{sec:smoothing}. Second, Nesterov's method
relies on a fixed domain and involves two projections onto the domain
at each iteration, whereas PRISMA uses only a single projection.
Moreover, Nesterov's methods and guarantees rely on the domain being
bounded and require that the number of iterations must be fixed in
advance, in order to set its parameters.  On the other hand, the
dynamic tuning of $\beta_k$ in PRISMA allows us to obtain a guarantee
that depends only on $\|x^*-x_1\|$, and with parameters
(i.e.~$\beta_k$ and $\theta_k$) which do not have to depend on it,
or on a fixed number of iterations.  Unbounded
domains are frequently encountered in practice, for example, in our
motivating application of max-norm optimization. The excessive gap
primal-dual algorithm in \citet{Nesterov05} improves over
the one in \citet{nesterov2005smooth} because it does not need to fix in advance
the number of iterations, but it shares the same shortcoming on the
assumption of the bounded domain.


The primal-dual algorithm in \citet{ChambolleP11} can minimize
functions of the form $g(K x)+h(x)$, where $K$ is a linear operator,
and it assumes to have access to $\prox_g$ and $\prox_h$. The original
problem is transformed into a saddle point problem and they obtain a
$\scO(\frac{1}{k})$ convergence rate for the primal-dual gap, for
non-smooth objective functions. However this kind of guarantee will
translate to a guarantee on the functional objective value only
assuming that $g$ and $h^*$ are Lipschitz (see also the discussion on
this point in \citet{LorisV11}). Moreover it is not clear how the
obtained bound explicitly depends on the Lipschitz constants and the
other relevant quantities.

A different approach is given by the Alternating Direction Method of Multipliers (ADMM) 
algorithms~\citep{BoydPCPE11}, that optimize the augmented Lagrangian associated to the optimization problem, updating alternatively the variables of the problem.
ADMM approaches have the drawback that constraints in the optimization problem are not necessarily satisfied at every iteration. This makes them ill suited to problems like max-norm matrix completion, where two matrices would be generated at each iteration, only one of the two being PSD, and for none of the two a convergence rate bound is known.
Recently, \citet{goldfarb,NIPS2010_0109} presented an ADMM method, called Alternating Linearization Method
(ALM), which uses alternate projections of an augmented Lagrangian. It is applicable when $g=0$, that is, for
objectives which can be decomposed into a smooth plus a ``simple''
function, and it has the same rate of convergence of FISTA. It can also be used with non-smooth functions, using the
smoothing method proposed by \citet{nesterov2005smooth}, and hence it shares the same shortcomings.

Another recent related method \citep{LorisV11} is applicable when $h=0$
and $f$ is quadratic.  This method differs significantly from both
PRISMA and Nesterov's aforementioned methods, in part because no
$\theta_k$ and no affine update of $y_k$ are required. The
applicability of \citet{LorisV11} is not as general as PRISMA (for
example it does not apply to max norm regularization or basis pursuit)
and it is not clear whether it can be extended to any smooth function
$f$.  Similar to PRISMA, it attains an $\scO\left(\frac{1}{k}\right)$
rate without requiring boundedness of the domain.


PRISMA combines the benefits of many of the above methods: we can
handle functions decomposable to three parts, as in the max-norm
optimization problem, with an unbounded domain, and without having
the parameters to depend on the distance to a minimizer nor to fix a priori the number of iterations, at the cost of
only an additional log-factor.
This is an important advantage in practice: setting parameters is
tricky and we typically do not know in advance how many iterations will be required.

\section{Applications}
\label{sec:app}

\subsection{Max Norm Regularized Matrix completion}
\label{sec:maxnorm}

The max norm (also known as the
$\gamma_{2:\ell_1\rightarrow\ell_{\infty}}$ norm) of a matrix $X$ is
given by $\|X\|_{max} = \inf_{X=UV\trans} \|U\|_{2,\infty}
\|V\|_{2,\infty}$, where $\|U\|_{2,\infty}$ is the maximal row norm of
$U$.  The max norm has been suggested as an effective regularizer for
matrix completion (i.e.~collaborative filtering) problems, with some
nice theoretical and empirical advantages over the more commonly used
nuclear norm (or trace norm) \cite{Srebro,RinaNatiCOLT,lee_maxnorm}.
It has also been recently suggested to use it in clustering problems,
where it was shown to have benefits over a nuclear-norm based
approach, as well as over spectral clustering \cite{jalali_maxnorm}.
However, the max norm remains much less commonly used than the nuclear
norm. One reason might be the lack of good optimization approaches for
the max norm. 
It has also been recently brought our
attention that \citet{Jaggi11} suggested an optimization approach for
$\min_{\|X\|_{max} \leq 1} f(X)$ which is similar in some ways to
these first order methods, and which requires $\scO\left(\frac{L_f}{k}\right)$
iterations, but where each iteration consists of solving a max-cut
type problem instead of an SVD---the practicality of such an approach
is not yet clear and as far as we know it has not been implemented.
Instead, the only practical approach we are aware of is a non-convex
proximal point (NCPP) approach, without any performance guarantees, and
relying on significant ``tweaking'' of learning rates
\cite{lee_maxnorm}.  Here we present the first practical first order
optimization method for max norm regularized problems.

To demonstrate the method, we focus on the matrix completion setting,
where we are given a subset $\Omega$ of observed entries of a matrix
$M \in \R^{m \times n}$ and would like to infer the unobserved entries.
Using max norm regularization, we have to solve the following optimization problem:
\begin{align}
\label{eq:mxnorm_opt_prob}
&\min_W \  \lambda \|W\|_{max} + \|P_\Omega(W) - P_\Omega(M)\|^2_F~.
\end{align}
where $P_{\Omega}$ is a projection onto the entries in $\Omega$
(zeroing out all other entries).  Expressing the max norm through its
semi-definite representation
we can rewrite \eqref{eq:mxnorm_opt_prob} as
\begin{align*}
\min_{A,B,X} && \big\{ \lambda \max diag \left(\begin{smallmatrix} A & X \\ X\trans & B
\end{smallmatrix}\right) +  \|P_\Omega(X) - P_\Omega(M)\|^2_F \\
&&\qquad + \delta_{\Sb^{m+n}_{+}}\left(\begin{smallmatrix} A & X \\ X\trans & B \end{smallmatrix}\right) \big\}~.
\end{align*}

Setting $f(X)= \|P_\Omega(X) -
P_\Omega(M)\|^2_F$, $g(X)=\lambda \max diag (X)$ and
$h(X)=\delta_{\Sb^{m+n}_{+}}(X)$, we can apply PRISMA.  The proximity
operator of $h$ corresponds to the projection onto the set of $(m+n)
\times (m+n)$ positive semidefinite matrices, that can be computed by
an eigenvalue decomposition and setting to zero all the negative
eigenvalues.  The proximity operator of $g$ amounts to thresholding
the diagonal elements of the matrix.

To demonstrate the applicability of PRISMA, we conducted runtime
experiments on the MovieLens 100K
Dataset\footnote{\url{http://www.grouplens.org/node/73}}, selecting an
equal and increasing number of users and movies with the most ratings.
We used training data equal to 25\% of the total number of entries,
selected uniformly at random among the ranked entries. We implemented PRISMA in MATLAB.
We stopped PRISMA when the relative difference
$\frac{\|X_{k+1}-X_k\|_F}{\|X_k\|_F}$ was less than $10^{-5}$.  The
computation bottleneck of the algorithm is the computation of the
eigenvalue decomposition at each step. So, to speed-up the algorithm,
we implemented the strategy to calculate only the top $c$ eigenvalues,
where $c$ is defined as $\min(m+n,(\# \textit{ eigenvalues } > 0
\textit{ at previous iteration})+1)$. If the smallest calculated
eigenvalue is positive, we increase $c$ by 5 and recalculate them,
otherwise we proceed with the projection step. This strategy is based
on the empirical observation that the algorithm produces solutions of
decreasing rank, and practically this procedure gives a big speed-up
to the algorithm.  This kind of strategy is common in the matrix
completion literature and RPCA, see e.g. \cite{goldfarb09}.  We
compared our algorithm to SDPT3 \cite{TutuncuTT03}, a state-of-the
package for semidefinite-quadratic-linear programming.  For simplicity
the parameter $\lambda$ was set the same in all the experiments, equal
to $0.2 |\Omega|$ that guarantees reasonable good generalization
performance for all the sizes considered. The parameter $a$ of PRISMA
was set to a rough estimate of the optimal value, that is, to
$\frac{\lambda \sqrt{|\Omega|}}{(m+n) \|P_\Omega(M)\|_F}$.  The
running times and the final value of the objective function are listed
in Table~\ref{table:results_max_norm}.  SDPT3, a highly optimized SDP
solver relying on interior point methods, indeed has very good
performance for relatively small problems.  But very quickly, the poor
scaling of interior point methods kick in, and the runtime increases
dramatically.  More importantly, SDPT3 ran out of memory and could not
run on problems of size 300x300 or larger, as did other SDP solvers we
tried.  On the other hand, the number of iterations required by PRISMA
to reach the required precision increases only mildly with the size of
the matrices and the increase in time is mainly due to the more
expensive SVDs.  We could run PRISMA on much larger problems,
including the full MovieLens 100K data set.
In Figure~\ref{fig:movielens100k} we report the objective value as
function of time (in seconds) for PRISMA on this data set.  The time
of each PRISMA iteration decreased over time, as the rank of the
iterate decreases, ranging from 30 seconds of the first iterations to
2 seconds of the last ones, yielding a total runtime of several hours
on the entire dataset.  In contrast, it is inconceivable to use
generic SDP methods for such data.  On the other extreme, non-convex
proximal-point (NCPP) methods are indeed much faster on this (and even
much larger) data sets, taking only about ten seconds to reach a
reasonable solution.  However, even running NCPP for many hours with
different step size settings and regardless of the allowed dimensionality, NCPP was not
able to reach the global minimum, and was always at least $4\%$ worst
than the solution found by PRISMA---the minimum objective value attained by
NCPP is also plotted in Figure~\ref{fig:movielens100k}.  PRISMA thus
fills a gap between the very pricey SDP methods that are practical
only on tiny data sets, and the large-scale NCPP methods which require
significant parameter tweaking and produce only rough (even if sometimes
satisfying) solutions.

\begin{table}[t]
\centering \footnotesize
\begin{tabular}{|c||c|c||c|c|c|}  \hline
                   & \multicolumn{2}{|c||}{SDPT3} & \multicolumn{3}{|c|}{PRISMA} \\ \hline
Matrix Size        & Cost    & Time  & Cost     & Iter. & Time \\ \hline
100x100            & 1078.76 & 13s   & 1079.00  & 7366  & 124s \\ \hline
150x150            & 2617.65 & 89s   & 2618.11  & 7866  & 218s \\ \hline
200x200            & 4795.76 & 364s  & 4796.97  & 8908  & 281s \\ \hline
250x250            & 7736.76 & 1323s & 7738.59  & 9524  & 404s \\ \hline
300x300            & --      & --    & 11406.39 & 9843  & 578s \\ \hline
\end{tabular}
\caption{Performance evaluation of PRISMA and SDPT3 on subsets of the MovieLens 100K dataset.}
\label{table:results_max_norm}
\end{table}

\begin{figure}[t]
\centering
\includegraphics[width=.6\linewidth]{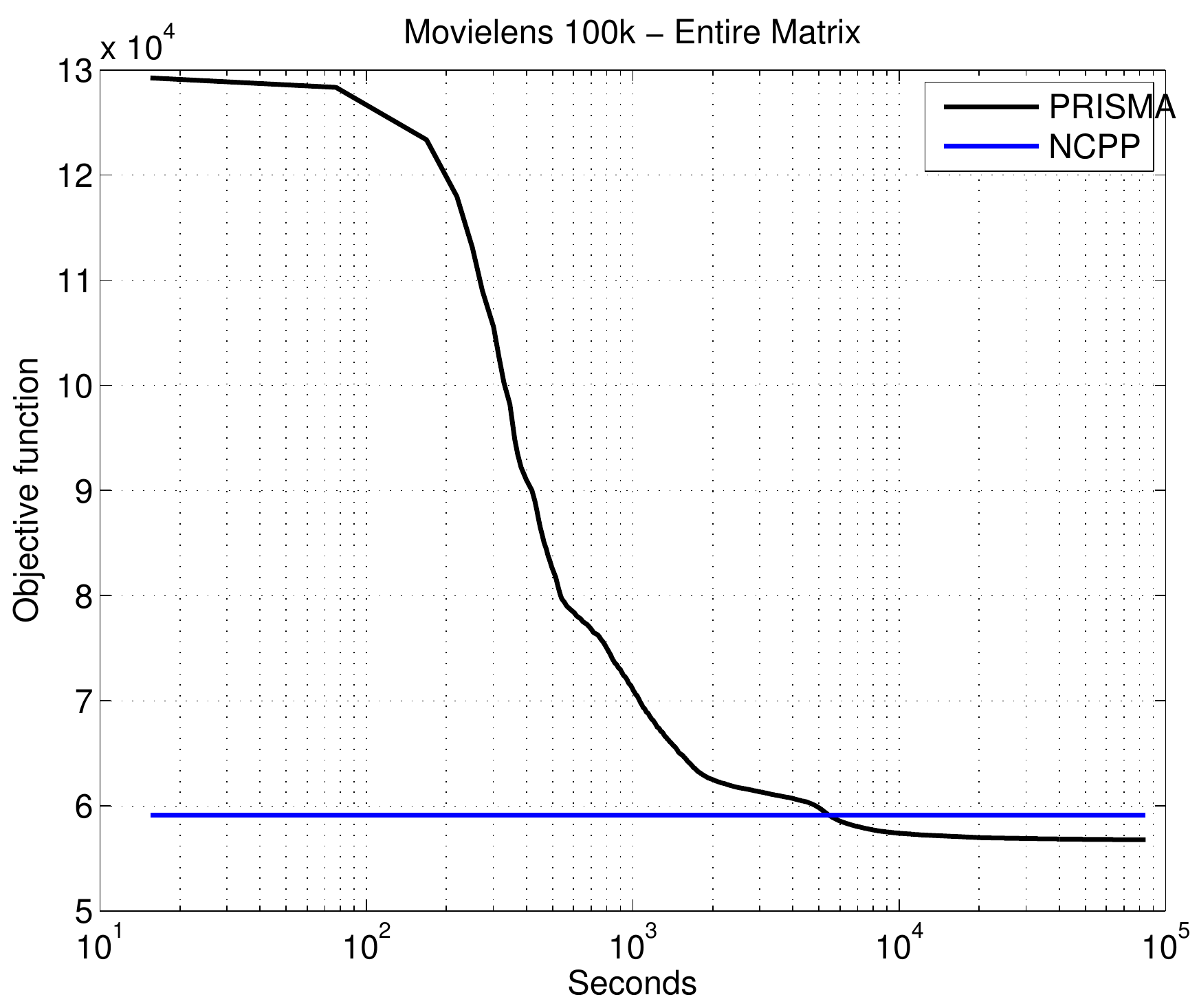}
\caption{Objective value for PRISMA as a function of iterations in log scales.}
\label{fig:movielens100k}
\end{figure}


\subsection{Robust PCA}
\label{sec:robustPCA}

In the method of \citet{robustPCA} for robust PCA, a data matrix $M
\in \R^{n_1 \times n_2}$ is given and the goal is to solve \beq \min \
\{ \|W\|_{tr} + \lambda \,\|S\|_1 : W + S = M, \, W,S\in \R^{n_1
  \times n_2} \}~.  \eeq This problem is of the form \eqref{eq:opt}
with the choice $f = 0, g(W) = \lambda \, \|M-W\|_1, h(W) =
\|W\|_{tr}$.  We used the RPCA method to solve the task of background
extraction from surveillance video.  By stacking the columns of each
frame into a long vector, we get a matrix $M$ whose columns correspond
to the sequence of frames of the video. This matrix $M$ can be
decomposed into the sum of two matrices $M = W + S$. The matrix $W$,
which represents the background in the frames, should be of low rank
due to the correlation between frames. The matrix $S$, which
represents the moving objects in the foreground in the frames, should
be sparse since these objects usually occupy a small portion of each
frame.

We have compared our MATLAB implementation of PRISMA to
ALM~\cite{goldfarb09,goldfarb}, where this problem is solved with an
augmented Lagrangian method, smoothing both the trace and the $\ell_1$
norm.  We further used the ALM continuation strategy described
in~\citet{goldfarb09}, that is usually beneficial in practice even
though no theoretical iteration count guarantees are known for it.
The parameter $a$ of PRISMA, analogously to max-norm experiments, is
set to $\frac{\lambda\sqrt{n_1 n_2}}{\|M\|_F}$. We also compared it to
the static strategy, setting $\beta_k$ to constant times $a$.

In our experiments, we used two
videos\footnote{\url{http://perception.i2r.a-star.edu.sg/bk_model/bk_index.html}}
introduced in \citet{Li_Huang_Gu_Tian_2004}, ``Hall'' and ''Campus''
with the same parameters used in~\citet{goldfarb09}.  Since for all
algorithms the complexity for each update is roughly the same, and
dominated by an SVD computation, we directly compare the number of
iterations in order to reduce the hardware/implementation dependence
of the experiments.  The average CPU time per iteration for the algorithms and datasets are reported in Table~\ref{table:time_rpca}.
The results are shown in Figure~\ref{fig:rpca}. On both image sequences
PRISMA outperforms the static algorithm, uniformly over the number of
iterations.  Notice also that for some constant settings the algorithm does
not converge to the optimum, as expected from Theorem~\ref{thm:fista}.
Moreover it also outperforms the state-of-the-art algorithm ALM.

\begin{figure}[h]
  \centering \footnotesize
  \includegraphics[width=0.60\linewidth]{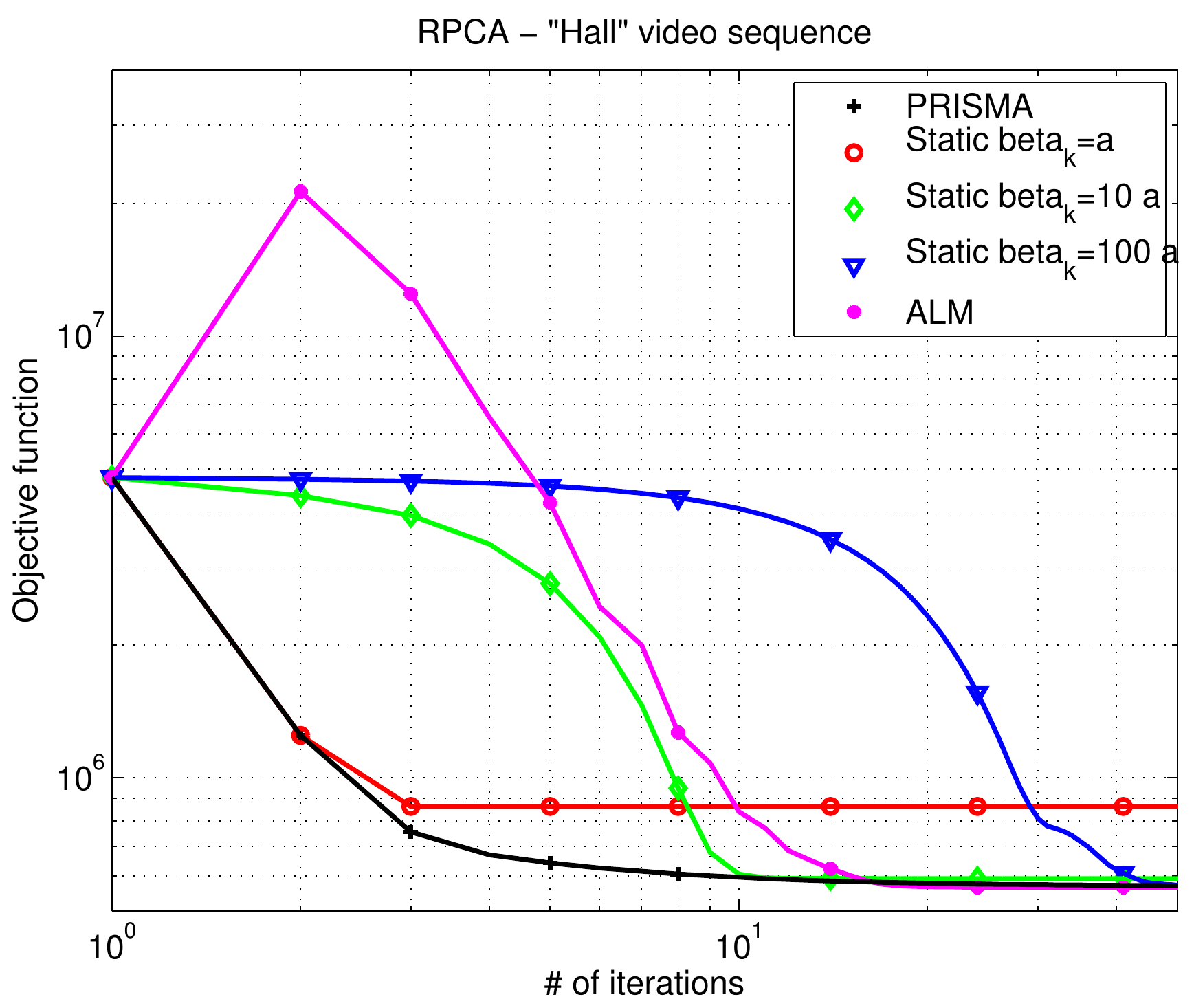} \\
  \includegraphics[width=0.60\linewidth]{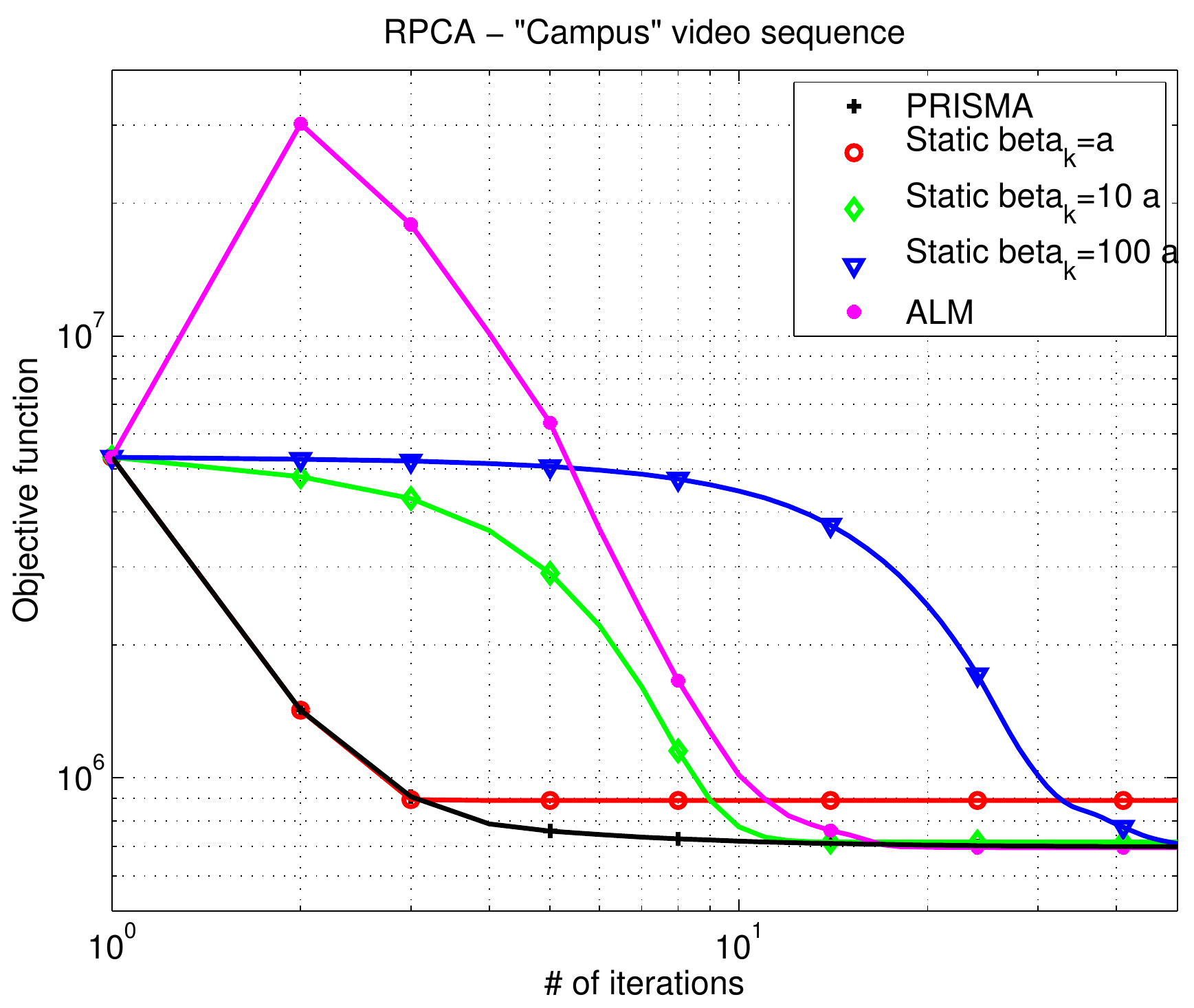}
\caption{Comparison of the objective values as a function of the number of iterations, in log scales.}
\label{fig:rpca}
\end{figure} 

\begin{table}[h]
\centering \footnotesize
\begin{tabular}{|c||c|c|}  \hline
       & Campus      & Hall         \\ \hline
PRISMA (static \& dynamic) & $3.5\pm0.1$ & $1.9\pm0.01$ \\ \hline
ALM    & $5\pm0.3$   & $2.66\pm0.1$ \\ \hline
\end{tabular}
\caption{Time (in seconds) per step, average $\pm$ standard deviation.}
\label{table:time_rpca}
\end{table}

 
\subsection{Sparse Inverse Covariance Selection}
\label{sec:l1}

In sparse inverse covariance selection (SICS) -- see \cite{wainwright_graph,yuan_graph} and references therein -- the goal is to solve 
\begin{equation}
\min_{X \in \Sb^n_{++}} \  -\log\det(X) + \lb \Sigma, X\rb + \lambda \, \|X\|_1,
\label{eq:spics}
\end{equation}
where $\Sigma \in \Sb^n_{+}$ is a given positive semidefinite matrix.
This problem is of the form \eqref{eq:opt} with the choice
$f = \lb \Sigma, \cdot\rb, g = \lambda \, \|\cdot\|_1$, 
$h = -\log\det(\cdot) + \delta_{\Sb^n_{++}}(\cdot)$.
The required proximity operator of $h$ at $X\in\Sb^n_{++}$ can be easily computed 
by computing a singular value decomposition of $X$ and solving 
a simple quadratic equation to obtain each singular value of the result.
We have compared PRISMA with two state-of-the-art approaches:
ALM~\cite{NIPS2010_0109}, which achieves stats-of-the-art empirical
results, and ADMM~\cite{BoydPCPE11}\footnote{We used the MATLAB code
  available at \url{http://www.stanford.edu/~boyd/papers/admm/}.}, a
recently popular general optimization method that performs very well
in practice, but lacks iteration complexity guarantees.  

ALM uses a complex continuation strategy that requires tuning of four
different parameters -- we used the data-set specific parameter
settings provided by~\citet{NIPS2010_0109}.  For ADMM we used the
default parameters.  For both ALM and ADMM there is no clear theory on
how to set their parameters.  To set the $a$ parameter of PRISMA we
need a rough estimate of $\|X^*\|_F$. The optimal solution of
\eqref{eq:spics} will be very close to a diagonal matrix, so we
approximate its optimum of adding the constraint $X=\alpha I$, and it
is easy to see that the optimal $\alpha$ is $\frac{1}{1+\lambda}$.
Using the fact that $\rho_g=\lambda n$, we set $a=\lambda (1+\lambda)
\sqrt{n}$ in all experiments.  That is, both for ADMM and PRISMA we
did not specifically tune any parameters.

We experimented on same five gene expression network data sets from
\citet{LiT10}, which were also used by \citet{NIPS2010_0109}, and used the same setting $\lambda=0.5$.

The results on the datasets are shown in Fig.~\ref{fig:sics}-\ref{fig:sics5}. Again we show the objective function in
relation to the number of iterations, being the complexity per step of
the algorithms roughly the same. The average times per iteration are reported
in Table~\ref{table:time_sics},\ref{table:time_sics2}. We can see that the performance of
PRISMA is pretty close to the one of ALM, while ADMM results in a
slower convergence.  Notice that the difference between ALM and ADMM
is probably mainly due to the continuation strategy used in ALM, whose
parameters are been tuned on these datasets. On the other hand PRISMA
just has one parameter, whose optimal setting is given by the theory.

\begin{figure}[h]
  \centering \footnotesize
  \includegraphics[width=0.60\linewidth]{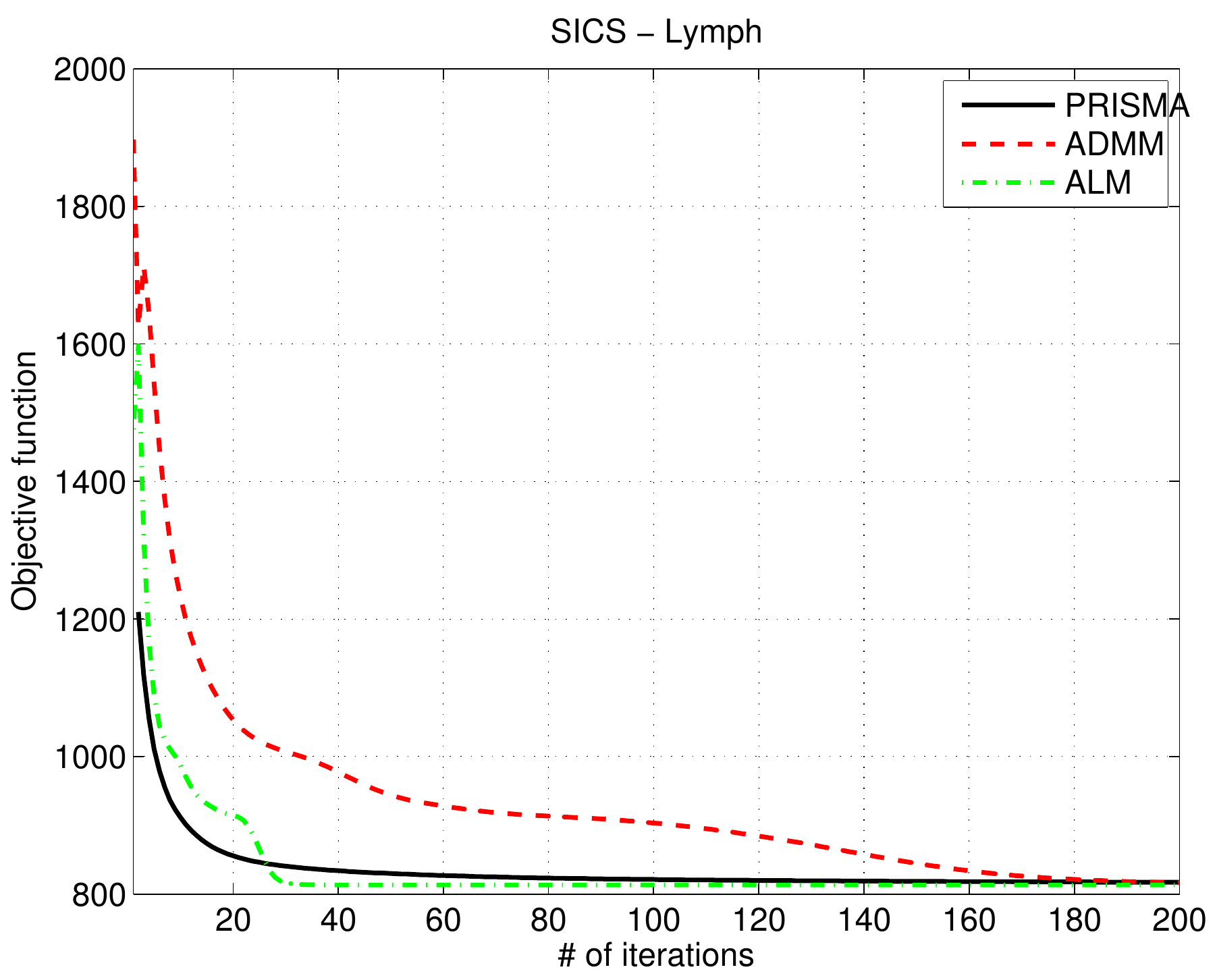}
\caption{Comparison of the objective values as a function of the iterations.}
\label{fig:sics}
\end{figure} 

\begin{figure}[h]
  \centering \footnotesize
  \includegraphics[width=0.60\linewidth]{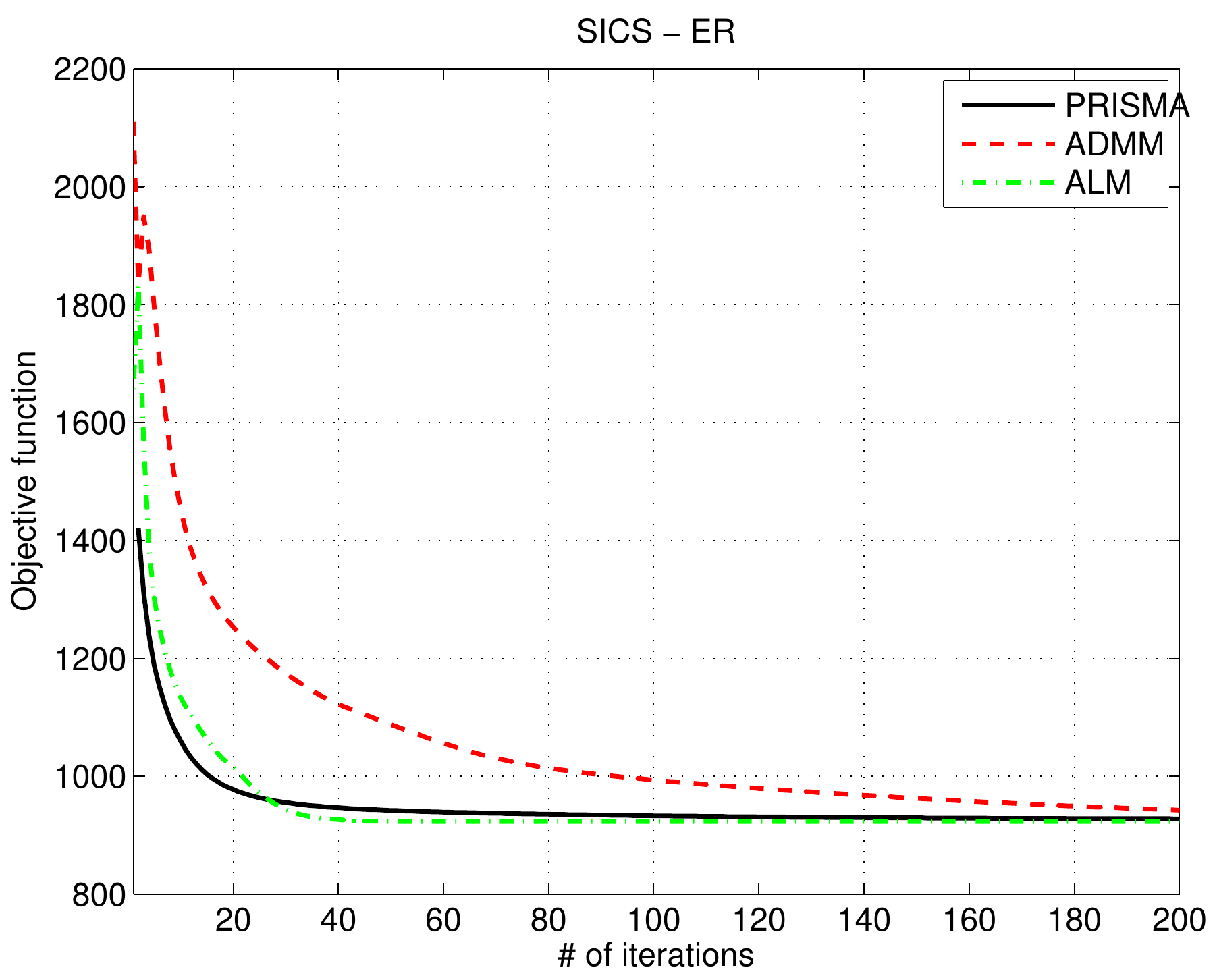}
\caption{Comparison of the objective values as a function of the iterations.}
\label{fig:sics2}
\end{figure}

\begin{figure}[h]
  \centering \footnotesize
  \includegraphics[width=0.60\linewidth]{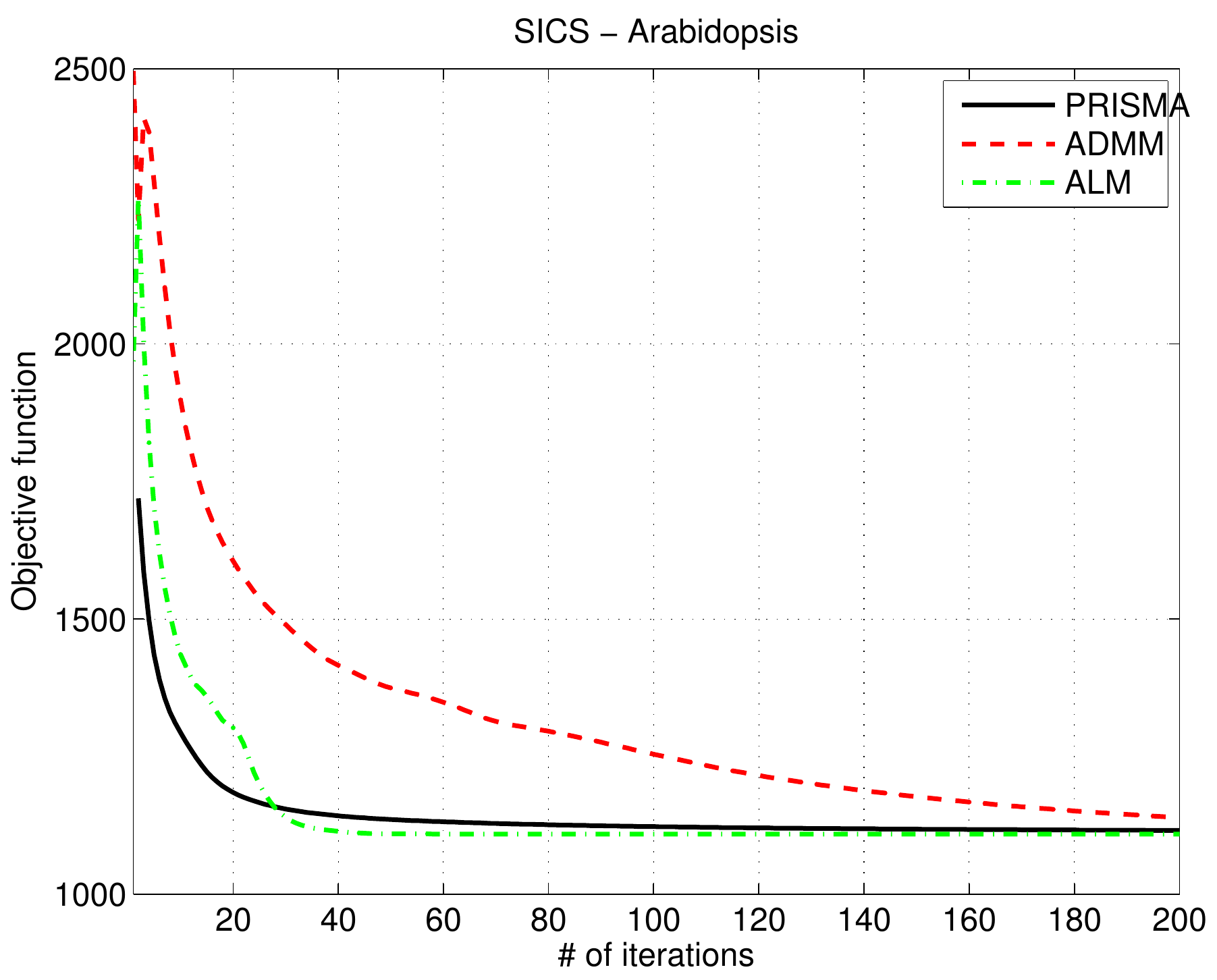}
\caption{Comparison of the objective values as a function of the iterations.}
\label{fig:sics3}
\end{figure} 

\begin{figure}[h]
  \centering \footnotesize
  \includegraphics[width=0.60\linewidth]{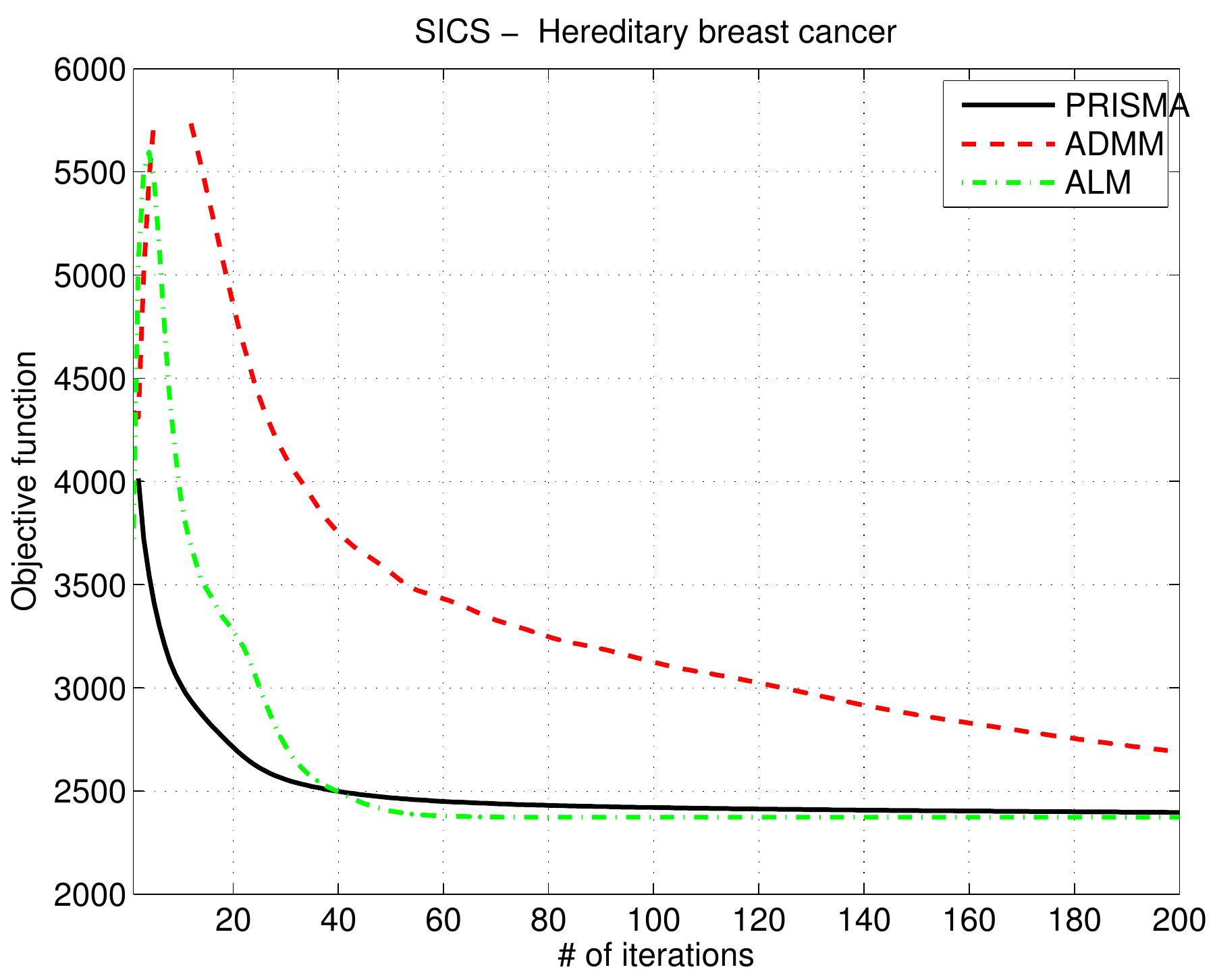}
\caption{Comparison of the objective values as a function of the iterations.}
\label{fig:sics4}
\end{figure}

\begin{figure}[h]
  \centering \footnotesize
  \includegraphics[width=0.60\linewidth]{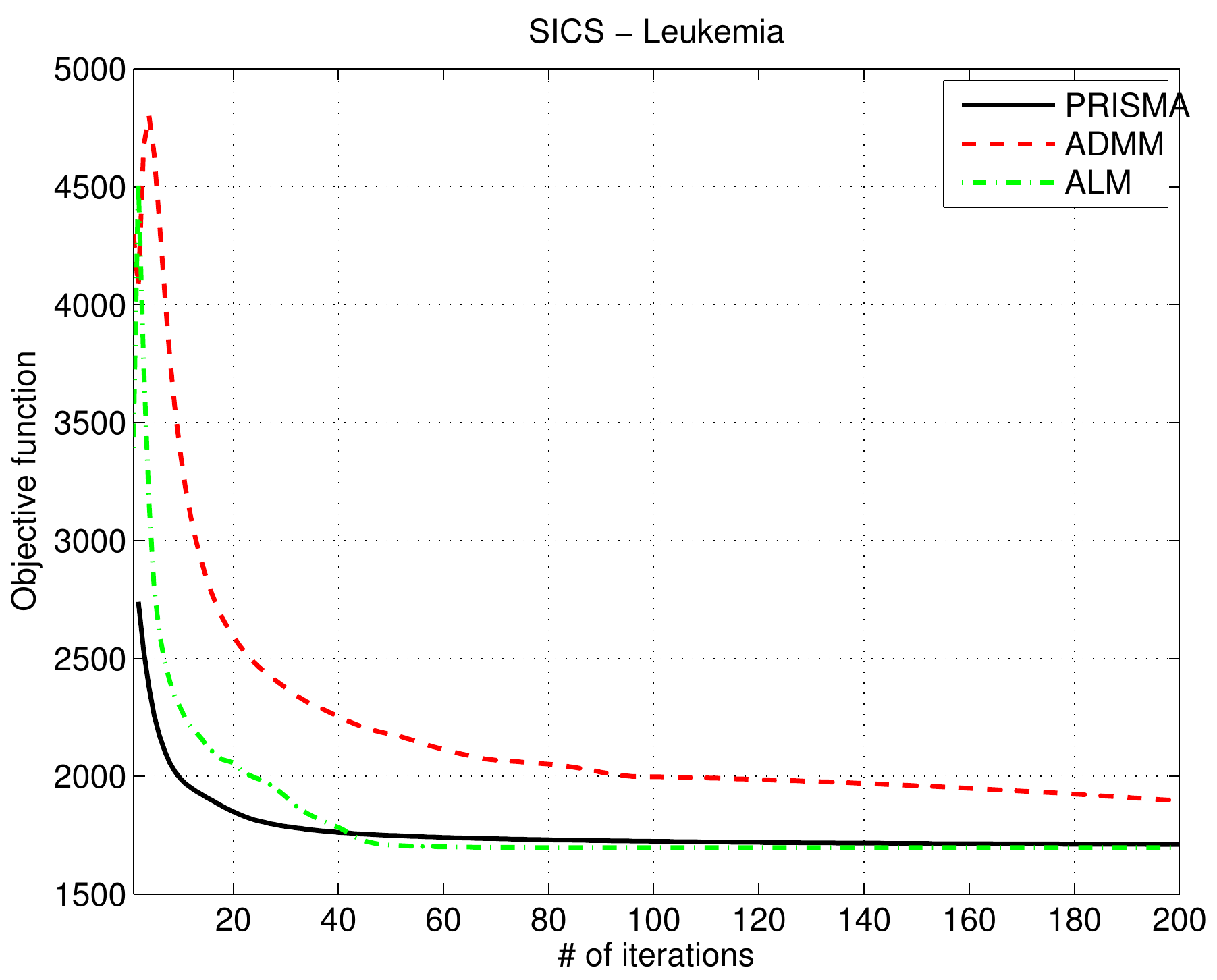}
\caption{Comparison of the objective values as a function of the iterations.}
\label{fig:sics5}
\end{figure}

\begin{table}[h]
\centering \footnotesize
\begin{tabular}{|c||c|c|}  \hline
       & Hereditary breast cancer &  Lymph        \\ \hline
PRISMA & $11.74\pm0.75$            & $0.33\pm0.02$ \\ \hline
ALM    & $7.5\pm0.09$             & $0.27\pm0.05$ \\ \hline
ADMM   & $13\pm0.23$            & $0.47\pm0.06$ \\ \hline
\end{tabular}
\caption{Time (in seconds) per step, average $\pm$ standard deviation.}
\label{table:time_sics}
\end{table}


\begin{table}[h]
\centering \footnotesize
\begin{tabular}{|c||c|c|c|}  \hline
       & Estrogen receptor & Arabidopsis thaliana & Leukemia      \\ \hline
PRISMA & $0.72\pm0.04$     & $1.15\pm0.09$        & $3.61\pm0.19$ \\ \hline
ALM    & $0.45\pm0.03$     & $0.75\pm0.04$        & $2.46\pm0.05$ \\ \hline
ADMM   & $0.80\pm0.02$     & $1.40\pm0.05$        & $4.15\pm0.09$\\ \hline
\end{tabular}
\caption{Time (in seconds) per step, average $\pm$ standard deviation.}
\label{table:time_sics2}
\end{table}

%
%

\section{Discussion and Future work}

We have proposed PRISMA, a new algorithm which can be used to solve many convex nonsmooth optimization problems
in machine learning, signal processing and other areas. PRISMA belongs to the broad family of
proximal splitting methods and extends in different ways forward-backward splitting, accelerated proximal,
smoothing and excessive gap methods. PRISMA is distinguished from these methods by its much wider 
applicability, which allows for solving problems such as basis pursuit or semidefinite programs such as 
max norm matrix completion.

We have validated our method with experiments on matrix completion,
robust PCA problems, and sparse inverse covariance selection, showing
that a simple to code implementation of PRISMA can handle large scale
problems and outperforms or be equal in efficiency to current state of
the art optimization algorithms.  But PRISMA can be useful also for
other learning-related problems.

For example, Basis Pursuit (BP)~\cite{basis_pursuit,kernel_bp} can also be
solved with PRISMA. The BP optimization problem is
formulated as
\begin{align}
\min \ \left\{\|x\|_1 : Ax = b , x\in \R^d \right\},
\end{align}
where $A\in \R^{m \times d}, b \in \R^m$ are prescribed input/output
data. In most applications, the sample size is much smaller than the
dimensionality, $m\ll d$. This problem is of the form \eqref{eq:opt}
with the choice $f=0, g=\|\cdot\|_1, h = \delta_\calA$, where $\calA =
\{ x\in\rd : Ax = b\}$.  
BP can be rephrased as a linear program, which is then solved
with standard approaches (simplex, interior point methods, etc.), or
it can be solved with ADMM, but in both cases no complexity results
are known.
It is easy to see that, for every $x\in\R^d$, its projection
on $\calA$ can be written as $\proj_{\calA}(x) = x-A\trans z$, where
$z$ is any solution of $A A\trans z = A x -b$, hence we can use PRISMA to solve this
problem. The complexity of the projection step is $\scO(d\,m)$, since $m\ll d$.
It would be interesting to see whether PRISMA would indeed yield empirical
advantages here, as well as in varied other learning and optimization
problems.


\bibliographystyle{plainnat}
\bibliography{lip_comp}

\end{document}